\documentclass[12pt,a4paper,leqno]{amsart}

\usepackage[latin1]{inputenc}
\usepackage[T1]{fontenc}
\usepackage{amsfonts}
\usepackage{amsmath}
\usepackage{amssymb}
\usepackage{eurosym}
\usepackage{mathrsfs}
\usepackage{palatino}
\usepackage{color}
\usepackage{esint}
\usepackage{url}
\usepackage{verbatim}

\usepackage{enumerate}

\newcommand{\R}{\mathbb{R}}
\newcommand{\C}{\mathbb{C}}

\newcommand{\Z}{\mathbb{Z}}

\newcommand{\diam}{\operatorname{diam}}
\numberwithin{equation}{section}

\newcommand{\dist}[0]{\operatorname{dist}}

\newcommand{\calD}[0]{\mathcal{D}}

\newcommand{\rest}{{\lfloor}}
\newcommand{\wt}[1]{{\widetilde{#1}}}

\swapnumbers
\theoremstyle{plain}
\newtheorem{thm}[equation]{Theorem}
\newtheorem{lem}[equation]{Lemma}
\newtheorem{prop}[equation]{Proposition}
\newtheorem{cor}[equation]{Corollary}

\theoremstyle{definition}
\newtheorem{defn}[equation]{Definition}

\theoremstyle{remark}
\newtheorem{rem}[equation]{Remark}

\title[Improved Cotlar's inequality]{Improved Cotlar's inequality in the context of local $Tb$ theorems}

\author{Henri Martikainen}
\address[H.M.]{Department of Mathematics and Statistics, University of Helsinki, P.O.B. 68, FI-00014 Helsinki, Finland}
\email{henri.martikainen@helsinki.fi}
\thanks{H.M. is supported by the Academy of Finland through the grant
Multiparameter dyadic harmonic analysis and probabilistic methods, and is a member of the Finnish Centre of Excellence in Analysis and Dynamics Research.}

\author{Mihalis Mourgoglou}
\address[M.M.]{Departament de Matem\`atiques, Universitat Aut\`onoma de Barcelona and Centre de Reserca Matem\` atica, Edifici C Facultat de Ci\`encies, 08193 Bellaterra (Barcelona)}
\email{mmourgoglou@crm.cat}
\thanks{Research of M.M. is supported by the ERC grant 320501 of the European Research Council (FP7/2007-2013).}

\author{Xavier Tolsa}
\address[X.T.]{ICREA and Departament de Matem\`atiques\\ Universitat Aut\`onoma de Barcelona \\ Edifici C Facultat de Ci\`encies\\
08193 Bellaterra (Barcelona)}
\email{xtolsa@mat.uab.cat}
\thanks{X.T. is supported by the ERC grant 320501 of the European Research Council (FP7/2007-2013), 2014-SGR-75 (Catalonia), MTM2013-44304-P (Spain), and by the Marie Curie ITN MAnET (FP7-607647). }

\makeatletter
\@namedef{subjclassname@2010}{%
  \textup{2010} Mathematics Subject Classification}
\makeatother


\begin{document}
\maketitle

\begin{abstract}
We prove in the context of local $Tb$ theorems with $L^p$ type testing conditions an improved version of Cotlar's inequality.
This is related to the problem of removing the so called buffer assumption of Hyt\"onen--Nazarov, which is the final barrier for the full
solution of S. Hofmann's problem. We also investigate the problem of extending the Hyt\"onen--Nazarov result to non-homogeneous measures.
We work not just with the Lebesgue measure but with measures $\mu$ in $\R^d$ satisfying $\mu(B(x,r)) \le Cr^n$, $n \in (0, d]$.
The range of exponents in the Cotlar type inequality depend on $n$. Without assuming buffer we get
the full range of exponents $p,q \in (1,2]$ for measures with $n \le 1$, and in general we get $p, q \in [2-\epsilon(n), 2]$, $\epsilon(n) > 0$.
Consequences for (non-homogeneous) local $Tb$ theorems are discussed.
\end{abstract}

\section{Introduction}
Let $\mu$ be a Radon measure on $\R^d$.
We say that a function $b_Q$ is an $L^p(\mu)$-admissible test function on a cube $Q \subset \R^d$ (with constant $B_1$), if
\begin{enumerate}
\item spt$\,b_Q \subset Q$,
\item $\mu(Q) = \int_Q b_Q\,d\mu$ and
\item $\int_Q |b_Q|^p\,d\mu \le B_1\mu(Q)$.
\end{enumerate}
A long standing problem (even for the Lebesgue measure $\mu = dx$) asks whether the $L^2$ boundedness
of a Calder\'on--Zygmund operator $T$ follows if we are given $p,q \in (1,\infty)$, and for every
cube $Q$ an $L^p(\mu)$-admissible test function $b_Q$ so that
$$
\int_Q |Tb_Q|^{q'}\,d\mu \lesssim \mu(Q)
$$
and an $L^q(\mu)$-admissible test function $p_Q$ so that
$$
\int_Q |T^*p_Q|^{p'}\,d\mu \lesssim \mu(Q).
$$

In the case that both exponents are simultaneously small, i.e. $p, q < 2$ (or even $p < 2 = q$), this is still not known in this original form.
However, Hyt\"onen--Nazarov \cite{HN} showed in the Lebesgue measure case that the $L^2$ boundedness follows
if one assumes the \emph{buffered} testing conditions
$$
\int_{2Q} |Tb_Q|^{q'}\,dx + \int_Q |T^*p_Q|^{p'}\,dx \lesssim |Q|.
$$
Notice that the estimate over $2Q$ is in fact equivalent to the same estimate over the whole space $\R^d$.
A key thing in the Lebesgue measure case is that if $1/p + 1/q \le 1$ (which includes the case $p=q=2$), then the original testing conditions automatically
imply the stronger buffered testing conditions by Hardy's inequality.
The non-homogeneous version for $p=q=2$ (without buffer) is by the first named author and Lacey \cite{LM:CZO}.

The need for the buffer assumption is related to delicate problems in passing from maximal truncations to the original operator. 
In the Hyt\"onen--Nazarov paper \cite{HN} the buffer is used in Lemma 3.2, which is a version of Cotlar's inequality
in the local $Tb$ setting (i.e. one needs to use the existence of the test functions to prove the Cotlar, not the boundeness
of the operator which one does not know). In this paper we prove a more sophisticated Cotlar's inequality (Theorem \ref{thm:Cot}),
which works in the non-homogeneous setting and (for the first time) always allows some exponents $p,q < 2$. For measures
satisfying $\mu(B(x,r)) \lesssim r$, the full range of exponents is obtained. This is our main result.

We also prove the related non-homogeneous local $Tb$ theorem with these improved exponents, which is Theorem \ref{thm:locTbop}. Here we choose to use the new strategy via the big pieces $Tb$
theorem and the good lambda method from the recent paper by the first two named authors and Vuorinen \cite{MMV}. In the Calder\'on--Zugmund realm
this technique currently requires antisymmetry.

The history of the various local $Tb$ theorems (not covered above) is extremely vast including the original one by M. Christ \cite{Ch} (with $L^{\infty}$ assumptions),
the non-homogeneous extension of this by Nazarov--Treil--Volberg \cite{NTVa} and the first one with $L^p$ testing conditions for model operators
by Auscher--Hofmann--Muscalu--Tao--Thiele \cite{AHMTT}. We also mention Auscher--Yang \cite{AY}, Auscher--Routin \cite{AR} and Hofmann \cite{Ho1}.
For a more extensive survey of the developments we refer to \cite{HN} and \cite{LM:CZO} (see also \cite{MMV}).

\section{Notation and definitions}
We say that a Radon $\mu$ on $\R^d$ is of degree $n \in (0, d]$ if for some constant $C_0 < \infty$ we have that $\mu(B(x,r)) \le C_0r^n$ for all $x \in \R^d$ and $r > 0$.

We say that $K \colon \R^d \times \R^d \setminus \{(x,y): x = y\} \to \C$ is an $n$-dimensional Calder\'on--Zygmund kernel if
for some $C < \infty$ and $\alpha \in (0,1]$ we have that
\begin{displaymath}
|K(x,y)| \le \frac{C}{|x-y|^n}, \qquad x \ne y,
\end{displaymath}
\begin{displaymath}
|K(x,y) - K(x',y)| \le C\frac{|x-x'|^{\alpha}}{|x-y|^{n+\alpha}}, \qquad |x-y| \ge 2|x-x'|,
\end{displaymath}
and
\begin{displaymath}
|K(x,y) - K(x,y')| \le C\frac{|y-y'|^{\alpha}}{|x-y|^{n+\alpha}}, \qquad |x-y| \ge 2|y-y'|.
\end{displaymath}
Given a Radon measure $\nu$ in $\R^d$, possibly complex, we define
\begin{displaymath}
T\nu(x) = \int K(x,y)\,d\nu(y), \qquad x \in \R^d \setminus \textup{spt}\,\nu.
\end{displaymath}
We also define $T\nu(x)$ as above for any $x \in \R^d$ whenever the integral on the right hand side makes sense. 
We say that $T$ is an $n$-dimensional SIO (singular integral operator) with kernel $K$. Since the integral may not always be absolutely convergent for $x \in \textup{spt}\,\nu$,
we consider the following $\epsilon$-truncated operators $T_{\epsilon}$, $\epsilon > 0$:
\begin{displaymath}
T_{\epsilon}\nu(x) = \int_{|x-y| > \epsilon} K(x,y)\,d\nu(y), \qquad x \in \R^d.
\end{displaymath}
The integral on the right hand side is absolutely convergent if, say, $|\nu|(\R^d) < \infty$.

For a positive Radon measure $\mu$ in $\R^d$ and $f \in L^1_{\textup{loc}}(\mu)$ we define
\begin{displaymath}
T_{\mu}f(x) = T(f\mu)(x), \qquad x \in \R^d \setminus \textup{spt}(f\mu),
\end{displaymath}
and
\begin{displaymath}
T_{\mu, \epsilon}f(x)  = T_{\epsilon}(f\mu)(x), \qquad x \in \R^d.
\end{displaymath}
The integral defining $T_{\mu, \epsilon}f(x)$ is absolutely convergent if for example $f \in L^p(\mu)$ for some $1 \le p < \infty$ and $\mu$ is of degree $n$.

We say that $T_{\mu}$ is bounded in $L^p(\mu)$ if the operators $T_{\mu, \epsilon}$ are bounded in $L^p(\mu)$ uniformly in $\epsilon > 0$.
Singular integral operators which are bounded in $L^2(\mu)$ are called Calder\'on--Zygmund operators (CZO). The boundedness of $T_{\mu}$
from $L^1(\mu)$ into $L^{1,\infty}(\mu)$ is defined analogously.

Let $M(\R^d)$ denote the space of finite complex Radon measures in $\R^d$
equipped with the norm of total variation $\|\nu\| = |\nu|(\R^d)$.
We say that $T$ is bounded from $M(\R^d)$ into $L^{1, \infty}(\mu)$ if there exists some constant $C < \infty$ so that for every $\nu \in M(\R^d)$ we have that
\begin{displaymath}
\sup_{\lambda > 0} \lambda\cdot \mu(\{x \in \R^d\colon\, |T_{\epsilon}\nu(x)| > \lambda\}) \le C\|\nu\|
\end{displaymath}
for all $\epsilon > 0$.

We still require the important concept of maximal truncations. If $T$ is an SIO then the maximal operator $T_*$ is defined by
\begin{displaymath}
T_{*}\nu(x) = \sup_{\epsilon > 0} |T_{\epsilon}\nu(x)|, \qquad \nu \in M(\R^d), \, x \in \R^d,
\end{displaymath}
and the $\delta$-truncated maximal operators $T_{*, \delta}$ is
\begin{displaymath}
T_{*, \delta}\nu(x) = \sup_{\epsilon > \delta} |T_{\epsilon}\nu(x)|, \qquad \nu \in M(\R^d), \, x \in \R^d.
\end{displaymath}
Like above, we also set
\begin{displaymath}
T_{\mu, *}f(x) = T_*(f\mu) \qquad \textup{and} \qquad T_{\mu, *, \delta}f(x) = T_{*,\delta}(f\mu).
\end{displaymath}

We need the following centred maximal functions with respect to balls and cubes:
\begin{displaymath}
M_{\mu}\nu(x) = \sup_{r > 0} \frac{|\nu|(B(x,r))}{\mu(B(x,r))}, \qquad M_{\mu}(f) := M_{\mu}(f\mu),
\end{displaymath}
and
\begin{displaymath}
M_{\mu}^{\mathcal{Q}}\nu(x) = \sup_{r > 0} \frac{|\nu|(Q(x,r))}{\mu(Q(x,r))}, \qquad M_{\mu}^{\mathcal{Q}}(f) := M_{\mu}^{\mathcal{Q}}(f\mu).
\end{displaymath}
The variant $M_{\mu, p}f := M_{\mu} ( |f|^p)^{1/p}$ will also be used.

A cube $Q \subset \R^d$ is said $\mu$-$(a,b)$-doubling (or just $(a,b)$-doubling if the measure $\mu$ is clear from the context) if
\begin{displaymath}
\mu(aQ) \le b\mu(Q),
\end{displaymath}
where $aQ$ is the cube concentric with $Q$ with diameter $a \diam(Q)$.
If $\mu$ is a measure of order $n$, then for $b > a^n$ we have the following result about the existence of doubling cubes.
For every $x \in \textup{spt}\,\mu$ and $c > 0$ there exist some $(a,b)$-doubling cube $Q$ centred at $x$ with $\ell(Q) \ge c$ (see Section 2.4 in \cite{ToBook}).

Given $t > 0$ we say that a cube $Q \subset \R^d$ has $t$-small boundary with respect to the measure $\mu$ if
\begin{displaymath}
\mu(\{x \in 5Q\colon\, \dist(x,\partial Q) \le \lambda\ell(Q)\}) \le t\lambda\mu(5Q)
\end{displaymath}
for every $\lambda > 0$ (here $\ell(Q)$ is the side length of $Q$). The following Lemma (Lemma 9.43 in \cite{ToBook}) is important for us (notice that it holds for general Radon measures).
\begin{lem}
Let $\mu_1$ and $\mu_2$ be two Radon measures on $\R^d$. Let $t > 0$ be some constant big enough (depending only on $d$). Then, given a cube $Q \subset \R^d$, there exists
a concentric cube $Q'$ so that $Q \subset Q' \subset 1.1Q$ which has $t$-small boundary with respect to $\mu_1$ and $\mu_2$.
\end{lem}

The final notation used is as follows.
We write $A \lesssim B$, if there is a constant $C>0$ so that $A \leq C B$. We may also write $A \sim B$ if $B \lesssim A \lesssim B$.
For a set $A$ we denote by $\mu \rest A$ the restriction of the measure
$\mu$ to the set $A$. \emph{All the appearing test functions are test functions with a uniform constant $B_1$} (as in the beginning of the Introduction).

\section{Cotlar's inequality}
The following is our improved version of Cotlar's inequality in the context of local $Tb$ theorems with $L^p$ type testing conditions.
Compare to the relatively simple Lemma 3.2 in \cite{HN} (this Lemma is the source of the buffer assumption in \cite{HN}).
The corollaries related to the integrability properties of the maximal truncations $T_{\mu,*}b_Q$ are discussed after the proof. 
\begin{thm}\label{thm:Cot}
Let $\mu$ be a measure of order $n$ on $\R^d$ and $T$ be an $n$-dimensional SIO. Let $b$ and $t$ be large enough constants (depending only on $d$).
Suppose $Q \subset \R^d$ is a fixed cube, $p, q \in (1,2]$ and $\delta > 0 $. We assume that there exists
an $L^p(\mu)$-admissible test function $b_Q$ in $Q$ so that
\begin{displaymath}
\int_Q |T_{\mu, \delta} b_Q|^{q'}\,d\mu \lesssim \mu(Q).
\end{displaymath}
Furthermore, we assume that
for every $(5,b)$-doubling cube $R \subset Q$ with $t$-small boundary there exists an $L^q(\mu)$-admissible test function
$p_R$ in $R$ so that
\begin{displaymath}
\left\{ \begin{array}{ll} 
\int_R |T_{\mu, \delta}^*p_R|^{p'}\,d\mu \lesssim \mu(R), & \textrm{if } \frac{1}{p}+\frac{1}{q} < 1 + \frac{1}{np}, \\
\int_{2R} |T_{\mu, \delta}^*p_R|^{p'}\,d\mu \lesssim \mu(R) & \textrm{otherwise }.
\end{array} \right.
\end{displaymath}
Then for every $\epsilon > \delta$ and $x \in (1-\tau)Q$, $\tau > 0$, we have that
\begin{displaymath}
|T_{\mu, \epsilon}b_Q(x)| \lesssim_{\tau} M_{\mu}b_Q(x) + M^{\mathcal{Q}}_{\mu, p} b_Q(x) +  M^{\mathcal{Q}}_{\mu, q'}(1_QT_{\mu, \delta}b_Q)(x).
\end{displaymath}
\end{thm}
\begin{proof}
Fix $\tau > 0$ and $x \in (1-\tau)Q$. Fix $\epsilon_0 > \delta$. Choose the smallest $m$ such that the
ball $B(x, 2^m \epsilon_0)$ is $(5C_d, b)$-doubling (where $C_d$ is a large enough dimensional constant), and let $\epsilon = 2^m\epsilon_0$. A standard calculation shows that
\begin{displaymath}
|T_{\mu, \epsilon_0}b_Q(x)-T_{\mu, \epsilon}b_Q(x)| \lesssim M_{\mu}b_Q(x).
\end{displaymath}

Therefore, it is enough to control $T_{\mu, \epsilon}b_Q(x)$. Suppose $\epsilon > C_d\ell(Q)$. Then we have that
\begin{displaymath}
T_{\mu, \epsilon}b_Q(x) = \int_{Q \cap B(x, \epsilon)^c} K(x,y)b_Q(y)\,d\mu(y) = \int_{\emptyset} K(x,y)b_Q(y)\,d\mu(y)  = 0.
\end{displaymath}
Suppose then that $c_{\tau}\ell(Q) \le \epsilon \le C_d\ell(Q)$. Then we have that
\begin{displaymath}
|T_{\mu, \epsilon}b_Q(x)| \lesssim_{\tau} \frac{1}{\ell(Q)^n} \int_{B(x, C_d\ell(Q))} |b_Q|\,d\mu \lesssim M_{\mu}b_Q(x).
\end{displaymath}
Finally, assume that $\epsilon < c_{\tau}\ell(Q)$ for a small enough constant $c_{\tau}$ to be fixed.
Define the Radon measure $\sigma_p = |b_Q|^p\,d\mu$.
Choose a cube $R$ centred at $x$ so that it has $t$-small boundary with respect to $\mu$ and $\sigma_p$, and
\begin{displaymath}
B(x, \epsilon) \subset R \subset B(x, C_d\epsilon) \subset Q.
\end{displaymath}
The last inclusion holds if $c_{\tau}$ is fixed small enough.
Notice that
\begin{displaymath}
\mu(5R) \le \mu(B(x, 5C_d\epsilon)) \le b\mu(B(x,\epsilon)) \le b\mu(R).
\end{displaymath}
Therefore, $R \subset Q$ is also a $\mu$-$(5, b)$-doubling cube. This means that there exists a function $p_R$ like in the assumptions.

For $z \in R$ we write
\begin{displaymath}
T_{\mu, \epsilon}b_Q(x) = T_{\mu, \epsilon}b_Q(x) - T_{\mu, \delta}(b_Q1_{(2R)^c})(z) + T_{\mu, \delta}b_Q(z) - T_{\mu, \delta}(b_Q1_{2R})(z).
\end{displaymath}
For all $z \in R$ we have that
\begin{align*}
|&T_{\mu, \epsilon}b_Q(x) - T_{\mu, \delta}(b_Q1_{(2R)^c})(z)| \\
&\le \int_{(2R)^c} |K(x,y)-K(z,y)|\,|b_Q(y)|\,d\mu(y) + \int_{B(x,\epsilon)^c \cap (2R)} |K(x,y)|\,|b_Q(y)|\,d\mu(y) \\
&\lesssim \epsilon^{\alpha} \int_{B(x,\epsilon)^c} \frac{|b_Q(y)|}{|x-y|^{n+\alpha}} \,d\mu(y) + \frac{1}{\epsilon^n} \int_{B(x, 2C_d\epsilon)} |b_Q(y)|\,d\mu(y) \lesssim M_{\mu}b_Q(x).
\end{align*}
We now estimate
\begin{align*}
|T_{\mu, \epsilon}b_Q(x)| &=  \Big| \frac{1}{\mu(R)} \int_R p_R(z) T_{\mu, \epsilon}b_Q(x)\,d\mu(z) \Big| \\
&\lesssim M_{\mu}b_Q(x) +  \frac{1}{\mu(R)} \int_R |p_R|\, |T_{\mu, \delta}b_Q|\,d\mu + \frac{1}{\mu(R)} \int_{2R} |T_{\mu, \delta}^*p_R| \, |b_Q|\,d\mu.
\end{align*}
We have that
\begin{align*}
\frac{1}{\mu(R)} \int_R |p_R|\, |T_{\mu, \delta}b_Q|\,d\mu &\le \Big( \frac{1}{\mu(R)} \int_R |p_R|^q \,d\mu\Big)^{1/q} \Big( \frac{1}{\mu(R)} \int_R 1_Q|T_{\mu, \delta}b_Q|^{q'}\,d\mu \Big)^{1/q'} \\
&\lesssim M^{\mathcal{Q}}_{\mu, q'}(1_QT_{\mu, \delta}b_Q)(x).
\end{align*}

It remains to estimate
\begin{displaymath}
\frac{1}{\mu(R)} \int_{2R} |T_{\mu, \delta}^*p_R| \, |b_Q|\,d\mu.
\end{displaymath}
Under the stronger assumption $ \int_{2R} |T_{\mu, \delta}^*p_R|^{p'}\,d\mu \lesssim \mu(R)$ we can simply estimate as follows:
\begin{align*}
\frac{1}{\mu(R)} \int_{2R} |T_{\mu, \delta}^*p_R| \, |b_Q|\,d\mu &\lesssim \Big(\frac{1}{\mu(R)} \int_{2R} |T_{\mu, \delta}^*p_R|^{p'}  \Big)^{1/{p'}} \Big( \frac{1}{\mu(2R)}  \int_{2R} |b_Q|^p\,d\mu \Big)^{1/p} \\
&\lesssim M^{\mathcal{Q}}_{\mu, p} b_Q(x).
\end{align*}
In the previous argument we used that $R$ is doubling.

Assume now only the weaker estimate $\int_{R} |T_{\mu, \delta}^*p_R|^{p'}\,d\mu \lesssim \mu(R)$. Then we write
\begin{displaymath}
\frac{1}{\mu(R)} \int_{2R} |T_{\mu, \delta}^*p_R| \, |b_Q|\,d\mu = \frac{1}{\mu(R)} \int_{R} |T_{\mu, \delta}^*p_R| \, |b_Q|\,d\mu + \frac{1}{\mu(R)} \int_{2R \setminus R} |T_{\mu, \delta}^*p_R| \, |b_Q|\,d\mu.
\end{displaymath}
The first term is dominated by $M^{\mathcal{Q}}_{\mu, p} b_Q(x)$ using H\"older's inequality like above. The second
term will be handled by a more tricky small boundaries trick (recall that $R$ has $t$-small boundary with respect to the measure $\sigma_p = |b_Q|^p\,d\mu$).
Denote also $\sigma = \sigma_1$ and $\nu_R = |p_R|\,d\mu$.

We begin by estimating
\begin{align*}
\int_{2R \setminus R} |T_{\mu, \delta}^*p_R| \, |b_Q|\,d\mu &= \int_{2R \setminus R} \bigg| \int_{y\colon\,|y-z| > \delta} K(y,z)p_R(y)\,d\mu(y)\bigg|\, d\sigma(z) \\
&\lesssim \int_R \int_{2R \setminus R} \frac{d\sigma(z)}{|z-y|^n} \,d\nu_R(y).
\end{align*}
Notice that $\mu(\partial R) = 0$ since $R$ has $t$-small boundary with respect to $\mu$. Fixing $y \in \textup{int}\,R$ we estimate
\begin{align*}
\int_{2R \setminus R} \frac{d\sigma(z)}{|z-y|^n} 
&\le \sum_{j=0}^{\infty} \int_{\{z \not \in R\colon 2^{-j}\diam(R) \le |z-y| \le 2^{-j+1}\diam(R)\}} \frac{d\sigma(z)}{|z-y|^n} \\
&\lesssim \sum_{j=0}^{\infty} (2^{-j}\ell(R))^{-n} \sigma(B(y, 2^{-j+1}\diam(R)) \setminus R) \\
&= \sum_{j=0}^{\infty} \sum_{P \in \calD_j(R)} 1_P(y) \ell(P)^{-n} \sigma(B(y, 2^{-j+1}\diam(R)) \setminus R) \\
&\le \mathop{\sum_{P \in \calD(R)}}_{5P \cap \partial R \ne \emptyset} \frac{\sigma(5P)}{\ell(P)^n} 1_P(y).
\end{align*}
This yields
\begin{align*}
\int_{2R \setminus R} |T_{\mu, \delta}^*p_R| \, |b_Q|\,d\mu &\lesssim \mathop{\sum_{P \in \calD(R)}}_{5P \cap \partial R \ne \emptyset} \frac{\sigma(5P)\nu_R(P)}{\ell(P)^n} \\
&= \mathop{\sum_{P \in \calD(R)}}_{5P \cap \partial R \ne \emptyset} \frac{\sigma(5P)}{\ell(P)^{n/2}} \eta_P^{-1/2} \cdot \frac{\nu_R(P)}{\ell(P)^{n/2}} \eta_P^{1/2} \\
&\lesssim \sum_{P \in \calD(R)} \frac{\nu_R(P)^2}{\ell(P)^{n}} \eta_P + \mathop{\sum_{P \in \calD(R)}}_{5P \cap \partial R \ne \emptyset} \frac{\sigma(5P)^2}{\ell(P)^{n}} \eta_P^{-1} = A + B.
\end{align*}
Here we choose
\begin{displaymath}
\eta_P = \Big(\frac{\ell(P)}{\ell(R)}\Big)^u M^{\mathcal{Q}}_{\mu, p} b_Q(x)
\end{displaymath}
for some yet to be fixed $u > 0$.

We begin by estimating the term $A$. We have
\begin{align*}
A &= M^{\mathcal{Q}}_{\mu, p} b_Q(x) \sum_{P \in \calD(R)} \frac{[ \int_P |p_R|\,d\mu ]^2 }{\ell(P)^n} \Big(\frac{\ell(P)}{\ell(R)}\Big)^u \\
&\le M^{\mathcal{Q}}_{\mu, p} b_Q(x) \sum_{P \in \calD(R)} \frac{[ \mu(P)^{1/q'} (\int_P |p_R|^q\,d\mu)^{1/q} ]^2 }{\ell(P)^n} \Big(\frac{\ell(P)}{\ell(R)}\Big)^u \\
&\lesssim M^{\mathcal{Q}}_{\mu, p} b_Q(x) \Big[ \int_R |p_R|^q\,d\mu \Big]^{2/q-1}  \sum_{P \in \calD(R)} \frac{\int_P |p_R|^q\,d\mu}{\ell(P)^{n(1-2/q')}} \Big(\frac{\ell(P)}{\ell(R)}\Big)^u \\
&\lesssim M^{\mathcal{Q}}_{\mu, p} b_Q(x) \frac{\mu(R)^{2/q-1}}{\ell(R)^{n(1-2/q')}} \int_R |p_R|^q\,d\mu  \sum_{k=0}^{\infty} 2^{k(n-2n/q'-u)} \\
&\lesssim M^{\mathcal{Q}}_{\mu, p} b_Q(x) \mu(R),
\end{align*}
provided that
\begin{equation}\label{eq:u1}
u > n - \frac{2n}{q'}.
\end{equation}

We then continue by estimating the term $B$. We have
\begin{align*}
B &= \frac{1}{M^{\mathcal{Q}}_{\mu, p} b_Q(x)}  \mathop{\sum_{P \in \calD(R)}}_{5P \cap \partial R \ne \emptyset} \frac{[\int_{5P} |b_Q|\,d\mu]^2}{\ell(P)^{n}}\Big(\frac{\ell(R)}{\ell(P)}\Big)^u \\
&\lesssim \frac{1}{M^{\mathcal{Q}}_{\mu, p} b_Q(x)}  \mathop{\sum_{P \in \calD(R)}}_{5P \cap \partial R \ne \emptyset} \frac{\sigma_p(5P)^{2/p}}{\ell(P)^{n(1-2/p')}} \Big(\frac{\ell(R)}{\ell(P)}\Big)^u \\
&= \frac{1}{M^{\mathcal{Q}}_{\mu, p} b_Q(x)} \frac{1}{\ell(R)^{n(1-2/p')}} \sum_{k=0}^{\infty} 2^{ku} 2^{kn(1-2/p')} \mathop{\sum_{P \in \calD_k(R)}}_{5P \cap \partial R \ne \emptyset} \sigma_p(5P)^{2/p}.
\end{align*}
With a fixed $k$ we estimate
\begin{align*}
\mathop{\sum_{P \in \calD_k(R)}}_{5P \cap \partial R \ne \emptyset} \sigma_p(5P)^{2/p} \le \Big( \mathop{\sum_{P \in \calD_k(R)}}_{5P \cap \partial R \ne \emptyset} \sigma_p(5P) \Big)^{2/p}
= \bigg( \int |b_Q|^p \Big[ \mathop{\sum_{P \in \calD_k(R)}}_{5P \cap \partial R \ne \emptyset} 1_{5P} \Big] \,d\mu \bigg)^{2/p}.
\end{align*}
Notice then that here
\begin{displaymath}
5P \subset \{y \in 5R\colon\, d(y, \partial R) \le C2^{-k}\ell(R) \}
\end{displaymath}
and
\begin{displaymath}
\sum_{P \in \calD_k(R)} 1_{5P} \lesssim 1.
\end{displaymath}
Using that $R$ has $t$-small boundary with respect to $\sigma_p$ we can now deduce that
\begin{displaymath}
\mathop{\sum_{P \in \calD_k(R)}}_{5P \cap \partial R \ne \emptyset} \sigma_p(5P)^{2/p} \lesssim \sigma_p( \{y \in 5R\colon\, d(y, \partial R) \le C2^{-k}\ell(R) \})^{2/p} \lesssim 2^{-2k/p}\sigma_p(5R)^{2/p}.
\end{displaymath}
Noticing that
\begin{align*}
\sigma_p(5R)^{2/p} \le \mu(5R)^{2/p} M^{\mathcal{Q}}_{\mu, p} b_Q(x)^2 \lesssim \ell(R)^{n(2/p-1)} \mu(R) M^{\mathcal{Q}}_{\mu, p} b_Q(x)^2
\end{align*}
this yields that
\begin{align*}
B \lesssim M^{\mathcal{Q}}_{\mu, p} b_Q(x) \mu(R) \sum_{k=0}^{\infty} 2^{k(-2/p+u+n-2n/p')} \lesssim M^{\mathcal{Q}}_{\mu, p} b_Q(x) \mu(R)
\end{align*}
provided that
\begin{equation}\label{eq:u2}
u < \frac{2}{p} + \frac{2n}{p'}-n.
\end{equation}
Assuming that the constant $u$ can be chosen appropriately we have proved that
\begin{displaymath}
\frac{1}{\mu(R)} \int_{2R \setminus R} |T_{\mu, \delta}^*p_R| \, |b_Q|\,d\mu \lesssim \frac{A+B}{\mu(R)} \lesssim M^{\mathcal{Q}}_{\mu, p} b_Q(x).
\end{displaymath}

We see from \eqref{eq:u1} and \eqref{eq:u2} that the constant $u$ can be chosen if
\begin{displaymath}
n - \frac{2n}{q'} < \frac{2}{p} + \frac{2n}{p'}-n.
\end{displaymath}
This amounts precisely to
\begin{displaymath}
\frac{1}{p}+\frac{1}{q} < 1 + \frac{1}{np}.
\end{displaymath}
\end{proof}
The main implication is that the maximal truncation $T_{\mu, *}b_Q$ still satisfies reasonable testing condtions.
\begin{cor}
Let $\mu$ be a measure of order $n$ on $\R^d$ and $T$ be an $n$-dimensional SIO. Let $b$ and $t$ be large enough constants (depending only on $d$).
Suppose $Q \subset \R^d$ is a fixed cube, $p, q \in (1,2]$ and $\delta > 0 $. We assume that there exists
an $L^p(\mu)$-admissible test function $b_Q$ in $Q$ so that
\begin{displaymath}
\int_Q |T_{\mu, \delta} b_Q|^{q'}\,d\mu \lesssim \mu(Q).
\end{displaymath}
Furthermore, we assume that
for every $(5,b)$-doubling cube $R \subset Q$ with $t$-small boundary there exists an $L^q(\mu)$-admissible test function
$p_R$ in $R$ so that
\begin{displaymath}
\left\{ \begin{array}{ll} 
\int_R |T_{\mu, \delta}^*p_R|^{p'}\,d\mu \lesssim \mu(R), & \textrm{if } \frac{1}{p}+\frac{1}{q} < 1 + \frac{1}{np}, \\
\int_{2R} |T_{\mu, \delta}^*p_R|^{p'}\,d\mu \lesssim \mu(R) & \textrm{otherwise }.
\end{array} \right.
\end{displaymath}
Let $\tau > 0$ and $0 < a < p$. We have that
$$
\int_{(1-\tau)Q} [T_{\mu, *, \delta} b_Q]^a\,d\mu \lesssim_{\tau, a} \mu(Q).
$$
\end{cor}
\begin{proof}
Using Theorem \ref{thm:Cot} we see that
\begin{align*}
\int_{(1-\tau)Q} [T_{\mu, *, \delta} b_Q]^a\,d\mu &\lesssim_{\tau} \int_Q [M_{\mu}b_Q]^a\,d\mu + \int_Q [M^{\mathcal{Q}}_{\mu, p} b_Q]^a\,d\mu + \int_Q [M^{\mathcal{Q}}_{\mu, q'}(1_QT_{\mu, \delta}b_Q)]^a\,d\mu \\
&= I + II + III.
\end{align*}
Notice that $I \lesssim \mu(Q)$, which can be seen by using H\"older's inequality with the exponent $p/a > 1$ and the $L^p(\mu)$ boundedness of $M_{\mu}$.

For the remaining terms $II$ and $III$ we shall use the inequality
\begin{equation}\label{eq:weak}
\int_Q |f|^a\,d\mu \le \frac{s}{s-a} \mu(Q)^{1-a/s} \|f\|_{L^{s, \infty}(\mu)}^a, \qquad s > a.
\end{equation}
Using \eqref{eq:weak} with $s = p > a$ we see that
\begin{align*}
II \lesssim_a \mu(Q)^{1-a/p} \|M^{\mathcal{Q}}_{\mu, p} b_Q\|_{L^{p, \infty}(\mu)}^a &= \mu(Q)^{1-a/p} \| M^{\mathcal{Q}}_{\mu}(|b_Q|^p) \|_{L^{1,\infty}(\mu)}^{a/p} \\
& \lesssim \mu(Q)^{1-a/p} \|b_Q\|_{L^p(\mu)}^a \lesssim \mu(Q),
\end{align*}
where we used that $M^{\mathcal{Q}}_{\mu}$ maps $L^1(\mu) \to L^{1,\infty}(\mu)$ boundedly.

Similarly, using \eqref{eq:weak} with $s = q' \ge 2 > a$ we see that
\begin{align*}
III \lesssim_a \mu(Q)^{1-a/q'} \| M^{\mathcal{Q}}_{\mu}(1_Q |T_{\mu, \delta} b_Q|^{q'})\|_{L^{1,\infty}(\mu)}^{a/q'} \lesssim \mu(Q)^{1-a/q'} \|1_Q T_{\mu, \delta} b_Q\|_{L^{q'}(\mu)}^a \lesssim \mu(Q).
\end{align*}
This ends the proof.
\end{proof}
In the following corollary we record the fully symmetric statement.
\begin{cor}\label{cor:main}
Let $\mu$ be a measure of order $n$ on $\R^d$ and $T$ be an $n$-dimensional SIO. Let $b$ and $t$ be large enough constants (depending only on $d$),
and $p, q \in (1,2]$. For every $(5,b)$-doubling cube $Q \subset \R^d$ with $t$-small boundary we assume that there exist
an $L^p(\mu)$-admissible test function $b_Q$ in $Q$ so that
\begin{displaymath}
\left\{ \begin{array}{ll} 
\sup_{\delta > 0} \int_Q |T_{\mu, \delta} b_Q|^{q'}\,d\mu \lesssim \mu(Q), & \textrm{if } \frac{1}{p}+\frac{1}{q} < 1 + \frac{1}{nq}, \\
\sup_{\delta > 0} \int_{2Q} |T_{\mu, \delta} b_Q|^{q'}\,d\mu \lesssim \mu(Q) & \textrm{otherwise },
\end{array} \right.
\end{displaymath}
and an $L^q(\mu)$-admissible test function $p_Q$ in $Q$ so that
\begin{displaymath}
\left\{ \begin{array}{ll} 
\sup_{\delta > 0} \int_Q |T_{\mu, \delta}^*p_Q|^{p'}\,d\mu \lesssim \mu(Q), & \textrm{if } \frac{1}{p}+\frac{1}{q} < 1 + \frac{1}{np}, \\
\sup_{\delta > 0} \int_{2Q} |T_{\mu, \delta}^*p_Q|^{p'}\,d\mu \lesssim \mu(Q) & \textrm{otherwise }.
\end{array} \right.
\end{displaymath}
Let $\tau > 0$.
Then for every $(5,b)$-doubling cube $Q \subset \R^d$ with $t$-small boundary we have
$$
\int_{(1-\tau)Q} [T_{\mu, *} b_Q]^a\,d\mu \lesssim_{\tau, a} \mu(Q), \qquad 0 < a < p,
$$
and
$$
\int_{(1-\tau)Q} [T^*_{\mu, *} p_Q]^a\,d\mu \lesssim_{\tau, a} \mu(Q), \qquad 0 < a < q.
$$
\end{cor}
\begin{rem}
Notice that if $n = 1$ the condition
$$
1/p + 1/q < 1 + 1/(np) = 1 + 1/p
$$
only says that $q > 1$ (and the symmetric condition only says that $p > 1$) yielding the full range of exponents without buffer.
In general, one can have both $p, q < 2$ simultaneously without buffer, showcasing
that $p = q = 2$ is not a threshold after which one is required to assume buffer.
\end{rem}

\section{Implications to local $Tb$ theorems}
It is easier to prove local $Tb$ theorems assuming conditions for maximal truncations $T_{\mu, *}b_Q$ rather than $T_{\mu}b_Q$. In fact, there is a tradeoff
here. One needs much weaker conditions on $T_{\mu, *}b_Q$ compared to $T_{\mu}b_Q$, but of course $T_{\mu, *}b_Q$ is a larger
object to begin with. Probably most convenient is to prove a local $Tb$ theorem assuming conditions on $T_{\mu, *}b_Q$, and then reduce the one what with operator testing to this
via Corollary \ref{cor:main}. The point of the maximal truncations is to allow suppression arguments. In our proof these suppression arguments
are hidden to the big pieces $Tb$ theorem (originally by Nazarov--Treil--Volberg \cite{NTV:Vit}) that we apply. The method of proof in \cite{HN} also
involves suppression (in a different way) and the proof is not directly applicable in the non-homogeneous situation.

We want to adapt the convenient strategy from the recent paper by the first two named authors and Vuorinen \cite{MMV}.
This new strategy via the big pieces $Tb$ theorem and non-homogeneous good lambda method is ideal in the square function
setting, since there is no duality and no maximal truncations in that context. Because the big pieces $Tb$ theorem seems to be challenging to extend to
all Calder\'on--Zygmund operators (it only currently works for antisymmetric ones), we also make the antisymmetry assumption here.

\subsection{Big pieces via maximal truncations}
The next Proposition (Proposition \ref{prop:TstarmainProp}) with testing assumptions about maximal truncations corresponds to Proposition 2.3 in \cite{MMV}.
The proof from that setting can be directly moved here, and as such one could make the assumptions as weak as in \cite{MMV}.
For the convenience of the reader we quickly reprove a less general statement here (if one is interested in as general a statement
as possible, just look at \cite{MMV}). This will be enough for deriving the local $Tb$ theorem with operator testing, which is our main focus here.
\begin{defn}\label{defn:random}
Given a cube $Q \subset \R^d$ we consider the following random dyadic grid.
For small notational convenience assume that $c_{Q} = 0$ (that is, $Q$ is centred at the origin). 
Let $N \in \Z$ be defined by the requirement $2^{N-3} \le \ell(Q) < 2^{N-2}$.
Consider the random square $Q^* = Q^*(w) = w + [-2^N,2^N)^n$, where
$w \in [-2^{N-1}, 2^{N-1})^d =: \Omega_N = \Omega$. The set $\Omega$ is equipped with the normalised Lebesgue measure $\mathbb{P}_N = \mathbb{P}$.
We define the grid $\mathcal{D}(w) := \mathcal{D}(Q^*(w))$ (the local dyadic grid generated by the cube $Q^*(w)$). Notice that
$Q \subset \alpha Q^*(w)$ for some $\alpha < 1$, and $\ell(Q) \sim \ell(Q^*(w))$.
\end{defn}

\begin{prop}\label{prop:TstarmainProp}
Let $\mu$ be a measure of order $n$ on $\R^d$ and $T$ be an $n$-dimensional SIO with a kernel $K$ satisfying $K(x,y) = -K(y,x)$.
Let $Q \subset \R^d$ be a fixed cube, $q \in (1,\infty)$ and $b_Q$ be an $L^q(\mu)$-admissible test function in $Q$ with constant $B_1$.
Then there exists a small constant $c_1 = c_1(q, B_1) > 0$ with the following property. If there exist $s > 0$ and an exceptional set $E_Q \subset \R^d$
so that $\int_{E_Q} |b_Q|\,d\mu \le c_1\int_Q |b_Q|\,d\mu$ and
\begin{equation}\label{eq:maxtrun}
\sup_{\lambda > 0} \lambda^s \mu(\{x \in Q \setminus E_Q\colon\, T_{\mu,*}b_Q(x) > \lambda\}) \le B_2\mu(Q) \textup{ for some } B_2 < \infty,
\end{equation}
then there exists $G_Q \subset Q \setminus E_Q$ so that $\mu(G_Q) \gtrsim \mu(Q)$ and
$T_{\mu\rest G_Q}\colon L^2(\mu\rest G_Q) \to L^2(\mu\rest G_Q)$ with a norm depending on the constants in the assumptions.
\end{prop}
\begin{proof}
We can assume that spt$\, \mu \subset Q$. Indeed, if we have proved the theorem for such measures, we can then apply it to $\mu \rest Q$.
Let us define the measure $\sigma$ by setting $\sigma(A) = \int_A |b_Q|\,d\mu$. Also, write $b_Q = |b_Q|\widehat b_Q$
using the polar decomposition, so that $|\widehat b_Q| = 1$. The big pieces $Tb$ theorem by Nazarov--Treil--Volberg (Theorem 5.1 in \cite{ToBook})
will be applied to the measure $\sigma$ and the $L^{\infty}$-function $\widehat b_Q$. (Notice that Theorem 5.1 in \cite{ToBook}
is stated for the Cauchy operator, but holds true for all antisymmetric CZO with the same proof. Moreover, the $L^1$ testing assumption there can directly be weakened to a weak
type testing condtion.)

We fix $w$, and write $\mathcal{D}(w) = \calD$. We also write $\calD_0 = \calD(0)$.
Let $\mathcal{A} = \mathcal{A}_w$ consist of the maximal dyadic cubes $R \in \calD$ for which
\begin{displaymath}
\Big| \int_R \widehat b_Q \,d\sigma \Big| < \eta \sigma(R),
\end{displaymath}
where $\eta := \frac{1}{2}B_1^{-1/q}$.
We set
\begin{displaymath}
T = T_w = \bigcup_{R \in \mathcal{A}} R \subset \R^d.
\end{displaymath} 
Notice that
\begin{align*}
\sigma(Q) = \int_Q |b_Q|\,d\mu \le B_1^{1/q}\mu(Q) = B_1^{1/q} \int_Q b_Q\,d\mu = B_1^{1/q} \int_Q \widehat b_Q \,d\sigma.
\end{align*}
Then estimate
\begin{displaymath}
 \int_Q \widehat b_Q \,d\sigma =  \int_{Q \setminus T} \widehat b_Q \,d\sigma + \sum_{R \in \mathcal{A}} \int_R  \widehat b_Q \,d\sigma
 \le \sigma(Q \setminus T) + \eta\sigma(Q).
\end{displaymath}
Since $\eta B_1^{1/q} = 1/2$ we conclude that
\begin{displaymath}
\sigma(Q) \le B_1^{1/q} \sigma(Q \setminus T)  + \frac{1}{2}\sigma(Q),
\end{displaymath}
and so
\begin{displaymath}
\sigma(Q) \le 2B_1^{1/q} [\sigma(Q) - \sigma(T)].
\end{displaymath}
From here we can read that
\begin{displaymath}
\sigma(T) \le \tau_0 \sigma(Q), \qquad \tau_0 := 1 - \frac{1}{2B_1^{1/q}} = 1-\eta < 1.
\end{displaymath}

Next, let $\mathcal{F}$ consist of the maximal dyadic cubes $R \in \calD_0$ for which
\begin{displaymath}
\int_R |b_Q|^q\,d\mu > C_0\mu(R)
\end{displaymath}
or
\begin{displaymath}
\sigma(R) < \delta \mu(R),
\end{displaymath}
where $C_0 := [16B_1\eta^{-1}]^{q'}$ and $\delta := \eta/16$.
Let $\mathcal{F}_1$ be the collection of maximal cubes $R \in \calD_0$ satisfying the first condition, and define $\mathcal{F}_2$ analogously.
Note that
\begin{displaymath}
\mu\Big( \bigcup_{R \in \mathcal{F}_1} R \Big) \le B_1C_0^{-1}\mu(Q),
\end{displaymath}
and so
\begin{align*}
\sigma\Big( \bigcup_{R \in \mathcal{F}_1} R \Big) &= \int_{\bigcup_{R \in \mathcal{F}_1} R } |b_Q|\,d\mu \\
&\le \mu\Big( \bigcup_{R \in \mathcal{F}_1} R \Big)^{1/q'} \Big( \int_Q |b_Q|^q\,d\mu\Big)^{1/q}  \\
&\le [B_1C_0^{-1}]^{1/q'}\mu(Q)^{1/q'} \cdot B_1^{1/q}\mu(Q)^{1/q} =  \delta\mu(Q) \le \delta\sigma(Q).
\end{align*}
Finally, we record that
\begin{displaymath}
\sigma\Big( \bigcup_{R \in \mathcal{F}_2} R \Big) = \sum_{R \in \mathcal{F}_2} \sigma(R) \le \delta \sum_{R \in \mathcal{F}_2} \mu(R) = \delta\mu\Big( \bigcup_{R \in \mathcal{F}_2} R \Big)
\le \delta \mu(Q) \le \delta\sigma(Q).
\end{displaymath}
We may conclude that the set
\begin{displaymath}
H_1 = \bigcup_{R \in \mathcal{F}} R
\end{displaymath}
satisfies  $\sigma(H_1) \le 2\delta\sigma(Q) = \frac{\eta}{8}\sigma(Q)$.

We now record the important property of the exceptional set $H_1$. Let $x \in Q \setminus H_1$.
For any $R \in \calD_0$ satisfying that $x \in R$ we have that
\begin{align*}
\delta \le \frac{\sigma(R)}{\mu(R)} = \frac{1}{\mu(R)} \int_R |b_Q|\,d\mu
\le  \Big( \frac{1}{\mu(R)} \int_R |b_Q|^q\,d\mu\Big)^{1/q} \le C_0^{1/q}.
\end{align*}
Letting $\ell(R) \to 0$ we conclude that for $\mu$-a.e. $x \in Q \setminus H_1$ we have $|b_Q(x)| \sim 1$.

We need another exceptional set $H_2$. To this end, let
\begin{displaymath}
p(x) = \sup_{r > 0} \frac{\sigma(B(x,r))}{r^n} = \sup_{r > 0} \frac{1}{r^n} \int_{B(x,r)} |b_Q|\,d\mu = M^{R}_{\mu} b_Q(x).
\end{displaymath}
For $p_0 > 0$ let $E_{p_0} = \{p \ge p_0\}$. Notice that
\begin{align*}
\mu(E_{p_0}) = \mu(\{M^{R}_{\mu} b_Q \ge p_0\}) \le \frac{1}{p_0^q} \int [M^{R}_{\mu} b_Q]^q \,d\mu \lesssim \frac{1}{p_0^q} \mu(Q),
\end{align*}
and so
\begin{displaymath}
\sigma(E_{p_0}) \le \mu(E_{p_0})^{1/q'} \Big( \int_Q |b_Q|^q\,d\mu\Big)^{1/q} \lesssim \frac{1}{p_0^{q-1}} \mu(Q) \le  \frac{1}{p_0^{q-1}} \sigma(Q).
\end{displaymath}
We fix $p_0$ so large that $\sigma(E_{p_0/2^n}) \le \frac{\eta}{8}\sigma(Q)$. For $x \in \{p > p_0\}$ define
\begin{displaymath}
r(x) = \sup\{r > 0\colon \sigma(B(x,r)) > p_0r^n\},
\end{displaymath}
and then set
\begin{displaymath}
H_2 := \bigcup_{x \in \{p > p_0\}} B(x,r(x)).
\end{displaymath}
It is clear that every ball $B_r$ with $\sigma(B_r) > p_0r^n$ satisfies $B_r \subset H_2$.
Notice that if $y \in H_2$, then there is $x \in \{p > p_0\}$ so that $y \in B(x,r(x))$, and so $\sigma(B(y, 2r(x)) \ge \sigma(B(x,r(x)) \ge p_0r(x)^n = p_02^{-n}[2r(x)]^n$. We conclude that
$H_2 \subset E_{p_0/2^n}$, and so $\sigma(H_2) \le \frac{\eta}{8}\sigma(Q)$. 

We can take $c_1 = \eta/8$ on the statement of the theorem. This means
that $\sigma(E_Q) \le \frac{\eta}{8}\sigma(Q)$.
Define now $H = H_1 \cup H_2 \cup E_Q$. The properties of $H$ are as follows:
\begin{enumerate}
\item We have $\sigma(H) \le \frac{\eta}{2}\sigma(Q)$, and so
$\sigma(H \cup T_w) \le \big(1-\frac{\eta}{2}\big)\sigma(Q) = \tau_1 \sigma(Q)$, $\tau_1 < 1$.
\item If $\sigma(B_r) > p_0r^n$, then $B_r \subset H$.
\item $|b_Q(x)| \sim 1$ for $\mu$-a.e. $x \in Q \setminus H$.
\end{enumerate}
We also have for every $\lambda > 0$ that
\begin{align*}
\lambda^s & \sigma(\{x \in Q \setminus H\colon\,  T_{\sigma,*} \widehat b_Q(x) > \lambda\}) \\
&=\lambda^s \sigma(\{x \in Q \setminus H\colon\,  T_{\mu,*} b_Q(x) > \lambda\}) \\
&= \lambda^s \int_{\{x \in Q \setminus H\colon\,  T_{\mu,*} b_Q(x) > \lambda\}} |b_Q|\,d\mu \\
&\lesssim \lambda^s\mu(\{x \in Q \setminus E_Q\colon\,  T_{\mu,*} b_Q(x) > \lambda\}) \le B_2\mu(Q) \lesssim \sigma(Q).
\end{align*}
Appealing to the big pieces global $Tb$ theorem by Nazarov--Treil--Volberg (Theorem 5.1 in \cite{ToBook})
with the measure $\sigma$ and the \emph{bounded} function $\widehat b_Q$ we find $G_Q \subset Q \setminus H \subset Q \setminus E_Q$ so that
$\sigma(G_Q) \gtrsim \sigma(Q)$ and
\begin{equation}\label{eq:sigmab}
\sup_{\epsilon > 0} \|1_{G_Q} T_{\sigma, \epsilon}f \|_{L^2(\sigma)} \lesssim \|f \|_{L^2(\sigma)}
\end{equation}
for every $f \in L^2(\sigma)$ satisfying spt$\,f \subset G_Q$.

Let $\epsilon > 0$. Suppose now that $g \in L^2(\mu)$ and spt$\,g \subset G_Q$. We apply Equation \eqref{eq:sigmab} with $f = g / |b_Q|$ (since
$G_Q \subset Q \setminus H$ we have $|b_Q| \sim 1$ on the support of $g$).
Notice that
\begin{displaymath}
\|1_{G_Q} T_{\sigma, \epsilon}(g/|b_Q|) \|_{L^2(\sigma)} = \|1_{G_Q} T_{\mu, \epsilon}g \|_{L^2(\sigma)} \gtrsim \|1_{G_Q} T_{\mu, \epsilon}g \|_{L^2(\mu)}
\end{displaymath}
so that
\begin{displaymath}
\|1_{G_Q} T_{\mu, \epsilon}g \|_{L^2(\mu)} \lesssim \|g/|b_Q| \|_{L^2(\sigma)} \lesssim \|g\|_{L^2(\mu)}.
\end{displaymath}
Since $\epsilon > 0$ was arbitrary this means precisely that $T_{\mu\rest G_Q}\colon L^2(\mu\rest G_Q) \to L^2(\mu\rest G_Q)$ boundedly.
Moreover, we have that
\begin{displaymath}
\mu(Q) \le \sigma(Q) \lesssim \sigma(G_Q) = \int_{G_Q} |b_Q| \,d\mu \lesssim \mu(G_Q).
\end{displaymath}
We are done.
\end{proof}
We record as a corollary a local $Tb$ theorem with maximal truncations testing. Again, this could be improved as in \cite{MMV}, but our main focus
is the local $Tb$ theorem on the next subsection (only the previous proposition is needed for that).
\begin{cor}
Let $\mu$ be a measure of order $n$ on $\R^d$ and $T$ be an $n$-dimensional SIO with a kernel $K$ satisfying $K(x,y) = -K(y,x)$.
Suppose $q \in (1,\infty)$, and let $b$ and $t$ be large enough constants (depending only on $d$).
We assume that to every $(5,b)$-doubling cube $Q \subset \R^d$ with $t$-small boundary there is associated
an $L^q(\mu)$-admissible test function $b_Q$ in $Q$ with constant $B_1$ such that
\begin{displaymath}
\sup_{\lambda > 0} \lambda^s \mu(\{x \in Q\colon\, T_{\mu,*}b_Q(x) > \lambda\}) \le B_2\mu(Q) \textup{ for some } B_2 < \infty \textup{ and } s > 0.
\end{displaymath}
Then $T_{\mu}\colon L^2(\mu) \to L^2(\mu)$ with a bound depending on the above constants.
\end{cor}
\begin{proof}
Fix an arbitrary $(5,b)$-doubling cube $Q \subset \R^d$ with $t$-small boundary. By the good lambda method (Theorem \ref{thm:thegoodlambda} and Remark \ref{rem:goodlambda}) it is enough to show
that there exists $G_Q \subset Q$ so that $\mu(G_Q) \gtrsim \mu(Q)$ and $T_{\mu\rest G_Q}\colon L^2(\mu\rest G_Q) \to L^2(\mu\rest G_Q)$.
By Proposition \ref{prop:TstarmainProp} this follows from the assumptions.
\end{proof}
\subsection{Local $Tb$ theorem with operator testing}
\begin{thm}\label{thm:locTbop}
Let $\mu$ be a measure of order $n$ on $\R^d$ and $T$ be an $n$-dimensional SIO with a kernel $K$ satisfying $K(x,y) = -K(y,x)$.
Suppose $q \in (1,2]$, and let $b$ and $t$ be large enough constants (depending only on $d$).
We assume that to every $(5,b)$-doubling cube $Q \subset \R^d$ with $t$-small boundary there is associated
an $L^q(\mu)$-admissible test function $b_Q$ in $Q$ with constant $B_1$ such that
\begin{displaymath}
\left\{ \begin{array}{ll} 
\sup_{\delta > 0} \int_Q |T_{\mu, \delta} b_Q|^{q'}\,d\mu \lesssim \mu(Q), & \textrm{if } \frac{1}{q} < \frac{1}{2}\Big(1+ \frac{1}{nq}\Big), \\
\sup_{\delta > 0} \int_{2Q} |T_{\mu, \delta} b_Q|^{q'}\,d\mu \lesssim \mu(Q) & \textrm{otherwise }.
\end{array} \right.
\end{displaymath}
Then $T_{\mu}\colon L^2(\mu) \to L^2(\mu)$ with a bound depending on the above constants.
\end{thm}
\begin{proof}
Fix an arbitrary $(5,b)$-doubling cube $Q \subset \R^d$ with $t$-small boundary. By the good lambda method (Theorem \ref{thm:thegoodlambda} and Remark \ref{rem:goodlambda}) it is enough to show
that there exists $G_Q \subset Q$ so that $\mu(G_Q) \gtrsim \mu(Q)$ and $T_{\mu\rest G_Q}\colon L^2(\mu\rest G_Q) \to L^2(\mu\rest G_Q)$.
By Proposition \ref{prop:TstarmainProp} it is enough to show that
\begin{displaymath}
\int_{Q \setminus E_Q} T_{\mu, *}b_Q\,d\mu \lesssim \mu(Q)
\end{displaymath}
for some set $E_Q \subset \R^d$ satisfying that $\int_{E_Q} |b_Q|\,d\mu \le c_1\int_Q |b_Q|\,d\mu$, where $c_1 = c_1(B_1,q) > 0$.

Let $E_Q = Q \setminus (1-\tau_0)Q$ for some $\tau_0 < 1$ large enough. Then we have, since $Q$ has $t$-small boundary and is doubling, that
\begin{displaymath}
\mu(E_Q) \le \epsilon(\tau_0) \mu(Q),
\end{displaymath}
where $\lim_{\tau_0 \to 1} \epsilon(\tau_0) = 0$. In particular, we have that
\begin{align*}
\int_{E_Q} |b_Q|\,d\mu \le \mu(E_Q)^{1/q'} \|b_Q\|_{L^q(\mu)} & \le \epsilon(\tau_0)^{1/q'}B_1^{1/q} \mu(Q) \\ &= \epsilon(\tau_0)^{1/q'}B_1^{1/q} \int_Q b_Q\,d\mu \\
& \le \epsilon(\tau_0)^{1/q'}B_1^{1/q} \int_Q |b_Q|\,d\mu
\le c_1   \int_Q |b_Q|\,d\mu
\end{align*}
provided that $\tau_0 = \tau_0(B_1, q) < 1$ is fixed close enough to $1$.

Now the estimate
\begin{displaymath}
\int_{Q \setminus E_Q} T_{\mu, *}b_Q\,d\mu = \int_{(1-\tau_0)Q} T_{\mu, *}b_Q\,d\mu \lesssim \mu(Q)
\end{displaymath}
follows from Corollary \ref{cor:main}, and we are done.
\end{proof}
\appendix
\section{Good lambda method with small boundaries}
We prove a version of Theorem 2.22 from \cite{ToBook}, which is weaker in the sense that we require only cubes with small boundaries.
\begin{thm} \label{thm:thegoodlambda}
Let $\mu$ be a Radon measure on $\R^d$ of degree $n$ and $T$ be an $n$-dimensional SIO. Let $b>0$ and $C_1$ be 
big enough (depending only on $d$) and let $\theta>0$.
Suppose that for every $(5,b)$-doubling cube $Q$ with $C_1$-small boundary there exists some subset $G_Q\subset Q$, with
$\mu(G_Q)\geq \theta\,\mu(Q)$, such that $T_*$ is bounded from $M(\R^d)$ to $L^{1,\infty}(\mu\rest G_Q)$, with norm
bounded uniformly on $Q$. Then $T_\mu$ is bounded in $L^p(\mu)$, for $1<p<\infty$, with its norm depending on $p$ and on the
preceding constants.
\end{thm}
\begin{rem}\label{rem:goodlambda}
One can also assume that $T_{\mu\rest G_Q}\colon L^2(\mu\rest G_Q) \to L^2(\mu\rest G_Q)$ with norm bounded uniformly on $Q$, since
then $T_*$ is bounded from $M(\R^d)$ to $L^{1,\infty}(\mu\rest G_Q)$ by standard results (see e.g. Theorem 2.21 in \cite{ToBook}).
\end{rem}

To prove Theorem \ref{thm:thegoodlambda} we will use a Whitney's decomposition of some open set. In the next lemma
we show the precise version of the required decomposition.
\begin{lem} \label{lem:whitney}
If $\Omega\subset\R^d$ is open, $\Omega\neq\R^d$, then $\Omega$
can be decomposed as
$$\Omega = \bigcup_{i\in I} Q_i, $$
where $Q_i$, $i\in I$, are closed dyadic cubes with disjoint interiors such
that for some constants $R>20$ and $D_0\geq1$ depending only on $d$ the following holds:
\begin{itemize}
\item[(i)] $10Q_i \subset \Omega$ for each $i\in I$.
\item[(ii)] $R Q_i \cap \Omega^{c} \neq \varnothing$ for each $i\in I$.
\item[(iii)] For each cube $Q_i$, there are at most $D_0$ cubes $Q_j$
such that $10Q_i \cap 10Q_j \neq \varnothing$. Further, for such cubes $Q_i$, $Q_j$, we have $\ell(Q_i)\approx
\ell(Q_j)$.
\end{itemize}
Moreover, if $\mu$ is a positive Radon measure on $\R^d$ and
$\mu(\Omega)<\infty$, there is a family of cubes $\{\wt Q_j\}_{j\in S}$, with $S\subset I$, so that
$Q_j\subset \wt Q_j\subset 1.1 Q_j$, satisfying the following:
\begin{itemize}
\item[(a)] Each cube $\wt Q_j$, $j\in S$, is $(9,2D_0)$-doubling and has $C_1$-small boundary.
\item[(b)] The cubes $\wt Q_j$, $j\in S$, are pairwise disjoint.
\item[(c)]
\begin{equation} \label{bqht22}
\mu\biggl( \,\bigcup_{j\in S} \wt Q_j \biggr) \geq \frac1{8D_0}\,
\mu(\Omega).
\end{equation}
\end{itemize}
\end{lem}

\begin{proof} Whitney's decomposition into dyadic cubes satisfying (i), (ii)
and (iii) is a well known result. 

To prove the existence of the family of $\{\wt Q_j\}_{j\in S}$, 
we denote by $ I_{db}\subset I$ the subfamily of the indices such that the cubes from 
$\{Q_i\}_{i\in I_{db}}$ are $(10,2D_0)$-doubling. Then notice that 
$$\mu(Q_j)< \frac1{2D_0} \mu(10Q_j)\qquad\mbox{if $j\in I\setminus I_{db}$}.$$
Since 
$$\sum_{j\in I}1_{10Q_j}\leq D_01_\Omega,$$
we deduce that
$$\sum_{j\in I\setminus I_{db}} \mu(Q_j) \leq \frac1{2D_0} \sum_{j\in I}\mu(10Q_j)\leq \frac12\mu(\Omega).$$
Thus, 
\begin{equation*}
\mu\biggl( \,\bigcup_{j\in I_{db}} Q_j \biggr) \ge \mu(\Omega) - \sum_{j\in I\setminus I_{db}} \mu(Q_j) \ge  \frac12\,
\mu(\Omega),
\end{equation*}
and we can choose a finite subcollection $I_{db}^1 \subset I_{db}$ so that 
\begin{equation}\label{eqdjy33}
\mu\biggl( \,\bigcup_{j\in I_{db}^1} Q_j \biggr) \geq \frac14\,
\mu(\Omega).
\end{equation}

By the covering lemma with triple cubes (see e.g. Theorem 2.1 in \cite{ToBook}), there exists a subfamily $S\subset I_{db}^1$ such that the cubes
 $\{2Q_j\}_{j\in S} $ are pairwise disjoint, and 
 $$\bigcup_{j\in I_{db}^1} Q_j\subset \bigcup_{j\in I_{db}^1} 2Q_j \subset \bigcup_{j\in S} 6Q_j.$$
For each $j\in S$, we consider a cube $\wt Q_j$ with
$Q_j\subset \wt Q_j \subset 1.1Q_j$ with a $C_1$-small boundary. Such a cube exists e.g. by Lemma 9.43 in \cite{ToBook}.

Clearly, the cubes $\wt Q_j$, $j\in S$,
 are pairwise disjoint by
construction. Further,
$$\mu(9\wt Q_j)\leq \mu(10 Q_j)\leq 2D_0\,\mu(Q_j)\leq 2D_0\,\mu(\wt Q_j).$$
This means that the cubes are $(9, 2D_0)$-doubling as claimed.
The proof of (c) is also easy, using (\ref{eqdjy33}) and the doubling doubling property of the cubes 
$\{Q_j\}_{j\in S}$:
\begin{align*}
\mu(\Omega)&\leq  4\, \mu\biggl( \,\bigcup_{j\in I_{db}^1} Q_j \biggr)
\leq 4 \,\mu\biggl( \,\bigcup_{j\in S} 6Q_j \biggr) \\
& \leq 4 \sum_{j\in S} \mu(6Q_j)
\leq 8D_0\,\sum_{j\in S} \mu(Q_j) \leq 8D_0\,\sum_{j\in S} \mu(\wt Q_j) = 8D_0 \mu\Big( \bigcup_{j \in S} \wt Q_j\Big).
\end{align*}
\end{proof}

\begin{proof}[Proof of Theorem \ref{thm:thegoodlambda}]
To prove the theorem we just have  to adapt the arguments in Theorem 2.22 from
\cite{ToBook} with very minor changes. Indeed, almost all changes reduce to replacing the cubes $Q_i$, $i\in S$, in the proof
of Theorem 2.22 from \cite{ToBook} by the cubes $\wt Q_i$, $i\in S$, from Lemma \ref{lem:whitney}
(with $\Omega\equiv\Omega_\lambda$),
and to replace the sum 
$\sum_{i\in I\setminus S}\mu(Q_i)$ appearing in various places of that proof by 
$$\mu\biggl(\Omega_\lambda\setminus \bigcup_{i\in S} \wt Q_i\biggr).$$
The details are omitted.
\end{proof}

\end{document}